\documentclass[11pt, reqno]{amsart}
\usepackage{amsfonts,latexsym,amsthm,amssymb,graphicx}
\usepackage[all]{xy}
\DeclareFontFamily{OT1}{rsfs}{}
\DeclareFontShape{OT1}{rsfs}{n}{it}{<-> rsfs10}{}
\DeclareMathAlphabet{\mathscr}{OT1}{rsfs}{n}{it}

\newtheorem{theorem}{Theorem}[section]
\newtheorem{lemma}[theorem]{Lemma}
\newtheorem{corol}[theorem]{Corollary}
\newtheorem{prop}[theorem]{Proposition}

\newtheorem{question}[theorem]{Question}

{\theoremstyle{definition} \newtheorem{defin}[theorem]{Definition}}
{\theoremstyle{remark} \newtheorem{remark}[theorem]{Remark}
\newtheorem{example}[theorem]{Example}

\newcommand{\Cbb}{{\mathbb{C}}}

\newcommand{\Pbb}{{\mathbb{P}}}

\newcommand{\Zbb}{{\mathbb{Z}}}

\title[Chern Classes for Locally Quasi-Homogeneous Free Divisors]{Chern Classes of Logarithmic Vector Fields for Locally Quasi-Homogeneous Free Divisors}
\author{Xia Liao}
\address{
KIAS
85 Hoegiro Dongdaemun-gu
Seoul 02455 Republic of Korea
}
\email{liao@kias.re.kr}

\begin{document}
\maketitle

\begin{abstract}
Let $X$ be a nonsingular complex projective variety and $D$ a locally quasi-homogeneous free divisor in $X$. In this paper we study a numerical relation between the Chern class of the sheaf of logarithmic derivations on $X$ with respect to $D$, and the Chern-Schwartz-MacPherson class of the complement of $D$ in $X$. Our result confirms a conjectural formula for these classes, at least after push-forward to projective space; it proves the full form of the conjecture for locally quasi-homogeneous free divisors in $\mathbb P^n$. The result generalizes several previously known results. For example, it recovers a formula of M. Mustata and H. Schenck for Chern classes for free hyperplane arrangements. Our main tools are Riemann-Roch and the logarithmic comparison theorem of Calderon-Moreno, Castro-Jimenez, Narvaez-Macarro, and David Mond. As a subproduct of the main argument, we also obtain a schematic Bertini statement for locally quasi-homogeneous divisors.
\end{abstract}

\section{introduction}
Let $X$ be a nonsingular variety defined over $\Cbb$, $D$ a reduced effective divisor and $U=X\smallsetminus D$ is the hypersurface complement. After a conjecture of Aluffi \cite{hyparr}, we raised the following question in \cite{liao}: 
\begin{question}\label{q}
In the Chow ring $A_*(X)$, under what conditions is the formula
\begin{equation}\label{formula}
c_{SM}(1_U)=c(\textup{Der}_X(-\log D))\cap [X]
\end{equation}
true?
\end{question}

The left hand side of the formula is the Chern-Schwartz-MacPherson class of the open subvariety $U$, and the right hand side is the total Chern class of the sheaf of logarithmic vector fields along $D$. In \cite{liao}, we have proven that in the case where $X$ is a complex surface (not necessarily a projective surface), the above formula is true if and only if the divisor $D$ has only quasi-homogeneous singularities. This result strongly overlaps a result by Adrian Langer in \cite{MR1971155} Proposition 6.1 and Corollary 6.2.

This paper is a continuation of the investigation for conditions which guarantee the above formula. The main result of this paper is:

\begin{theorem}\label{main}
Let $X \subseteq \Pbb^N$ be a nonsingular projective variety defined over $\Cbb$, and let $D$ be a locally quasi-homogeneous free divisor in $X$. Then the two classes appearing in \eqref{formula} agree after push-forward to $\Pbb^N$, that is:
\begin{equation}\label{num}
\int_X \! c_1(\mathscr{O}(1))^i \cap c_{SM}(1_U) = \int_X \! c_1(\mathscr{O}(1))^i \cap \Big(c(\textup{Der}_X(-\log D))\cap [X]\Big)
\end{equation}
for all $i \geq 0$.
\end{theorem}

\begin{corol}\label{prspace}
If $X= \Pbb^n$, formula~\eqref{formula} holds in the Chow ring for all locally quasi-homogeneous free divisors.
\end{corol}

As a special case when $i=0$, both sides of formula ~\eqref{num} give the topological Euler characteristic of $U$. For a general $i$, this formula may be seen as an analogue for Chern classes of the computation of the cohomology of $U$ done by Castro et al. in \cite{MR1363009}.

In \cite{MR2928931}, Aluffi proves that formula ~\eqref{formula} is true for free arrangements of hypersurfaces that are locally analytically isomorphic to hyperplane arrangements; this includes the case of normal crossing divisors and of free hyperplane arrangements, verified earlier (\cite{MR1697199}, \cite{hyparr}). Aluffi's results along with the results in \cite{liao} have given convincing evidences why the freeness condition has to enter the picture when comparing these two classes.

It should be noted that in the context of hyperplane arrangement, the Chern polynomial of the sheaf of logarithmic vector fields has been compared with other naturally defined polynomials by several authors. Musta{\c{t}\u{a}} and Schenck related the Poincar\'{e} polynomial of a free arrangement to $c_t(\Omega^1_X(\log D))$, the Chern polynomial of the sheaf of logarithmic differentials\cite{MR1843320}. Aluffi related the characteristic polynomial of an arrangement to the Chern-Shwartz-MacPherson class of the arrangement complement, as well as the class in the Grothendieck ring of varieties $K_0(Var)$ of this complement\cite{hyparr}. The Poincar\'{e} polynomial and the characteristic polynomial of an arrangement are closely related\cite{MR1217488}. All these polynomials carry essentially the same information for free hyperplane arrangements. By Corollary~\ref{prspace}, the result in this paper recovers Aluffi's result for hyperplane arrangements \cite{hyparr}, and generalizes it to arbitrary quasi-homogeneous free hypersurfaces arrangements in projective space. Taking into account the relation between Poincar\'{e} polynomial and CSM class polynomial stated above, the result from this paper also gives a generalization of Musta{\c{t}\u{a}} and Schenck's result (Theorem 4.1 in \cite{MR1843320}) independent of the original proof.

While the freeness of the divisor is a nice relevant condition, it alone does not guarantee the truth of formula ~\eqref{formula}, as pointed out by our previous study\cite{liao}. The additional quasi-homogeneity condition on the divisor gives further control over the singularities. One famous result is: the Tjurina number being equal to the Milnor number characterizes quasi-homogeneity for isolated hypersurface singularities\cite{MR0294699}. More generally, The quasi-homogeneity of divisors can be described by certain properties of the Tjurina Algebras\cite{MR1385285}. In this paper, the quasi-homogeneity, together with the freeness of the divisor, are utilized to apply the logarithmic comparison theorem (LCT)\cite{MR1363009}, which implies the classes in ~\eqref{formula} have the same degree (the case $i=0$ in \eqref{num}).

As a result of our local analysis, we obtain a schematic version of Bertini's theorem for locally quasi-homogeneous divisors (see Corollary \ref{bertini}), strengthening Teissier's idealistic Bertini theorem (\cite{MR0568901} section 2.8) in this context. 

The remaining of this paper will prove theorem ~\ref{main} along the following roadmap:

\begin{enumerate}
\item With the help of Riemann-Roch and LCT, we show the classes appearing in ~\eqref{formula} have the same degree.
\item We show a result of Bertini type: the freeness and the quasi-homogeneity of the pair $(X, D)$ are preserved by intersecting a general hyperplane $H$. So this allows us to use LCT inductively on the new pair $(X \cap H,D \cap H)$.
\item We study the relation between $\textup{Der}_{X \cap H}(-\log D \cap H)$ and $c_1(\mathscr{O}(1)) \cdot c(\textup{Der}_X(-\log D))$, using an exact sequence which relates the sheaf of the logarithmic vector fields and the normal bundle. We also show $c_{SM}(1_U)$ behaves in a similar fashion.
\item We conclude our proof by an induction on the dimension of $X$.

\end{enumerate}
We will stick to the complex analytic category for all our discussions in this paper. Occasionally we may state a result for schemes. I am very grateful to J\"org Sch\"urmann, who showed me the connection between LCT and question \ref{q}. I feel greatly indebted to my advisor Paolo Aluffi, whose constant help and warm encouragement finally lead to this  work. I am also very grateful to the anonymous referee for the detailed comments on local analytic geometry and many good opinions on improving the presentation of this paper.

\section{The first step of the proof}
Let $D$ be a quasi-homogeneous free divisor in a complex algebraic projective variety $X$ of dimension $n$, and $U$ be the hypersurface complement $X \smallsetminus D$. In this section we prove $c_{SM}(1_U)$ and $c(\textup{Der}_X(-\log D))\cap [X]$, as cycle classes in the Chow ring of $X$, have the same degree. 

For any complex algebraic variety $X$, we can assign to it the group $C(X)$ of constructible functions on $X$ and the group $A(X)$ of algebraic cycles modulo the rational equivalence relation. Both assignments are functorial with respect to proper morphisms of varieties. There is a unique natural transformation (called MaPherson transformation) between these functors, taking the indicator function $1_X$ to the total Chern class of the tangent bundle $c(TX)\cap [X]$ for any nonsingular $X$. In this paper, the image of the function $1_U$ under this transformation is denoted by $c_{SM}(1_U)$. For more detailed information on Chern-Schwartz-MacPherson classes, see MacPherson's original paper \cite{MR0361141} and Example 19.1.7 in \cite{MR732620}.

We briefly recall the definition of free divisors and some of their basic properties\cite{MR586450}. Over an open subset $U$ (in the complex topology), let $f$ be a defining equation of the divisor $D$. In this open set, a logarithmic vector field along $D$ is a derivation $\theta \in \textup{Der}_X(U)$ satisfying $\theta f \in (f)$, and a logarithmic $p$-form with poles along $D$ is a meromorphic $p$-form $\omega \in \Omega_X^p(\star D)(U)$ satisfying $f \cdot \omega \in \Omega_X^p(U)$ and $f \cdot \mathrm{d} \omega \in \Omega_X^p(U)$.  This local information globalizes into coherent $\mathcal{O}_X$-modules $\textup{Der}_X(-\log D)$ and $\Omega_X^p(\log D)$. There is a natural dual pairing between $\textup{Der}_X(-\log D)$ and $\Omega_X^1(\log D)$, making them reflexive with respect to each other. A divisor is free if $\textup{Der}_X(-\log D)$ (and so $\Omega_X^1(\log D)$) is locally free. In this case, its rank equals $\dim X = n$.

According to the second of the defining properties of the logarithmic p-forms along $D$, the exterior differential makes the sheaves of logarithmic forms a complex $\Omega_X^{\bullet}(\log D)$. It is a subcomplex of $\Omega_X^{\bullet}(\star D)$, the complex of the sheaves of meromorphic forms with poles along $D$.  When $D$ is free, the sheaf of logarithmic p-forms is isomorphic to the pth exterior product of the sheaf of logarithmic 1-forms: $\Omega_X^p(\log D) \cong \Lambda^p\Omega_X^1(\log D)$.

Next we recall the definition of locally quasi-homogeneous divisors\cite{MR2500865} (See also \cite{MR1363009}).

\begin{defin}
A germ of divisor $(D,p) \subset (X,p)$ is quasi-homogeneous if there are local coordinates $x_1, \ldots , x_n \in \mathcal{O}_{X,p}$  with respect to which $(D,p)$ has a weighted homogeneous defining equation (with strictly positive weights). We also say $D$ is quasi-homogenous at $p$. A divisor $D$ in an $n$-dimensional complex manifold $X$ is locally quasi-homogeneous if the germ $(D,p)$ is quasi-homogeneous for each point $p \in D$. A germ of divisor $(D,p) \subset (X,p)$ is locally quasi-homogeneous if the divisor $D$ is locally quasi-homogeneous in a neighborhood of $p$.
\end{defin}

\begin{example}[\cite{MR1363009}]
Consider the surface $D \subset \Cbb^3$ given by $f(x,y,z) = x^5z + x^3y^3 + y^5z$. The germ of divisor $(D,0)$ is quasi-homogeneous with respect to the weight (1,1,1). However, $(D,0)$ is not locally quasi-homogeneous. In fact, $(D,p)$ is not quasi-homogeneous for any point $p$ on the $z$-axis other than the origin.
\end{example}

One benefit of considering locally quasi-homogenous free divisors is, that this class of divisors enjoys the following logarithmic comparison theorem (LCT):

\begin{theorem}[\cite{MR1363009}, see also \cite{MR2500865} for generalizations]\label{LCT}
Let $D$ be a locally quasi-homogeneous free divisor in the complex manifold $X$. Then the inclusion of complexes $\Omega_X^{\bullet}(\log D) \to \Omega_X^{\bullet}(\star D)$ is a quasi-isomorphism.
\end{theorem}

\begin{remark}\label{app}
Denote $j: U \to X$ the inclusion. The Grothendieck comparison theorem states the De Rham morphism:
\begin{equation*}
\Omega_X^{\bullet}(\star D) \to \mathbf{R}j_* \mathbf{C}_U
\end{equation*} 
is a quasi-isomorphism. Grothendieck's result together with theorem ~\ref{LCT} imply that, for locally quasi-homogeneous free divisors, the logarithmic De Rham complex computes the cohomology of $U$.
\end{remark}

\begin{remark}
In \cite{MR2500865}, it is pointed out that the Jacobian ideal $J_D$ of linear type, with the freeness of the divisor, is enough to imply LCT. This condition, which is purely algebraic, may be used for generalizing theorem ~\ref{main} in the future.
\end{remark}

Now, we can prove:
\begin{prop}\label{degree}
For locally quasi-homogeneous free divisors in nonsingular projective complex varieties, $c_{SM}(1_U)$ and $c(\textup{Der}_X(-\log D))\cap [X]$ have the same degree.
\end{prop}

\begin{proof}
By the functoriality of Chern-Schwartz-MacPherson transformation, we have: 

\begin{equation*}
\int_X \! c_{SM}(1_U) = \chi_c(U)
\end{equation*} 
where $\int_X$ denotes degree of the dimension ~$0$ component of the given class.

In the context of complex algebraic varieties, or compactifiable complex analytic manifold, it is well known that the Euler characteristic with compact support equals the usual Euler characteristic (\cite{MR2031639} Proposition 2.0.2), and the latter is computed by the logarithmic De Rham complex by remark ~\ref{app}:

\begin{equation*}
\begin{split}
\chi_c(U) &= \chi(U)\\
                 &= \sum_i (-1)^i H^i \left(\mathbf{R}\Gamma(X; \Omega_X^{\bullet}(\log D) \right)\\
                 &= \sum_{p,q} (-1)^{p+q} H^p (X; \Omega_X^q(\log D))\\
                 &= \sum_q (-1)^q \int_X \! \mathrm{ch}(\Omega_X^q(\log D)) \cap \mathrm{Td}(X)\\
                 &= \int_X \! \sum_q (-1)^q \mathrm{ch}(\Lambda^q \Omega_X^1(\log D)) \cap \mathrm{Td}(X)\\
                 &= \int_X \! c_n(\textup{Der}_X(-\log D)) \cdot \mathrm{Td}(\textup{Der}_X(-\log D))^{-1} \cap \mathrm{Td}(X)\\
                 &= \int_X \! c_n(\textup{Der}_X(-\log D)) \cap [X]\\
                 &= \int_X \! c(\textup{Der}_X(-\log D)) \cap [X]\\
\end{split}
\end{equation*}
The fourth of these equalities comes from Riemann-Roch; the fifth one comes from the fact that $\Omega_X^p(\log D) \cong \Lambda^p\Omega_X^1(\log D)$ for free divisors; the sixth uses a formula showed in \cite{MR732620} Example 3.2.5; for the seventh, just notice that $c_n(\textup{Der}_X(-\log D)) \cap [X]$ is a 0-dimensional cycle and the Todd class starts from $1+\ldots$ for any vector bundle.

\end{proof}

\section{The second step of the proof}
In this section we prove the freeness and the quasi-homogeneity of the pair $(X,D)$ is preserved by intersecting with a general hyperplane of the ambient projective space $\Pbb^N$. For this purpose, we first recall Saito's logarithmic stratification (see \cite{MR586450} lemma 3.2), which we will use in our definition of transversal intersections.

\begin{lemma}
Let $X$ be an $n$-dimensional complex manifold and $D$ a divisor in it. There exists uniquely a stratification $\{D_\alpha, \alpha \in I\}$ of $X$ with the following properties:

\begin{enumerate}
\item Stratum $D_\alpha$, $\alpha \in I$ is a smooth connected immersed submanifold of $X$. $X$ is a disjoint union $\coprod_{\alpha \in I}D_\alpha$ of the strata.
\item let $p \in X$ belong to a stratum $D_\alpha$. Then the tangent space $T_{D_\alpha,p}$ of $D_\alpha$ at $p$ coincides with the subspace $\textup{Der}_X(-\log D)(p) \subset T_{X,p}$.
\item if $D_\alpha \cap \overline{D}_\beta \neq \emptyset$, then $D_\alpha \subset \partial D_\beta$.
\end{enumerate} 
\end{lemma}

In the lemma, each $D_\alpha$ is in fact a maximal integral submanifold of the involutive distribution determined by $\textup{Der}_X(-\log D)$. With the help of these logarithmic stratum, we can now define transversal intersection of $D$ with nonsingular varieties.

\begin{defin}
Let $X$ and $Y$ be complex analytic submanifolds of some ambient space, and $D$ is a divisor in $X$. We say that the divisor $D$ intersects $Y$ transversally at $p \in D \cap Y$ if $T_{D_\alpha,p}$ intersects $T_{Y,p}$ transversally where $D_\alpha$ is the unique stratum containing $p$. We say that the divisor $D$ intersects $Y$ transversally if at all points of $D \cap Y$ they intersect transversally.
\end{defin}

\begin{remark}
E. Faber has recently studied transversality for singular hypersurfaces in \cite{faber}. It turns out that our definition of transversal intersection overlaps her definition of ``splayed divisors" when one ``splayed component" is a nonsigular hypersurface.
\end{remark}

The following propositions show that quasi-homogeneity and freeness are preserved by transversal intersection. 

\begin{prop}\label{quasi}
Let $X \subset \Pbb^N$ be an $n$-dimensional nonsingular complex projective variety and $D$ a locally quasi-homogeneous divisor in $X$. Let $H$ be a hyperplane of $\Pbb^N$ intersecting both $D$ and $X$ transversally. Then $D \cap H$ is a locally quasi-homogeneous divisor in the nonsingular $X \cap H$.
\end{prop}

\begin{proof}
At any point $p \in H \cap D$, the transversal intersection of $D$ and $H$ indicates that we can find a non-vanishing logarithmic vector field $\theta$ at a neighborhood of $p$ such that $H$ is transversal to $\theta_p$ at $p$. By \cite{MR1363009} Lemma 2.3 there exists a coordinate system locally at $p$ such that $(D,p) \cong (D' \times \mathbb{C}, (0,0))$ and $(D',0)$ is a quasi-homogeneous germ in $\mathbb{C}^{n-1}$. In this coordinate system, $H \cap X$ becomes a nonsingular hypersurface $H'$ transversal to the trivial factor. Projecting along the trivial factor gives a local isomorphism between the pairs $(\mathbb{C}^{n-1}, D') \cong (H', (D' \times \mathbb{C}) \cap H')$. 
\end{proof}

\begin{prop}\label{free}
Let $X \subset \Pbb^N$ be an $n$-dimensional nonsingular complex projective variety, $D$ a free divisor in $X$, and $H$ a hyperplane of $\Pbb^N$ intersecting both $D$ and $X$ transversally. Then $D \cap H$ is a free divisor in the nonsingular $X \cap H$.
\end{prop}

\begin{proof}
cf. \cite{MR1363009} Lemma 2.2.
\end{proof}

\begin{remark}
The projective space $\Pbb^N$ and the hyperplane $H$ play no role in the local analysis. The propositions will remain true if we replace $\mathbb{P}^N$ by any nonsingular ambient space, and replace $H$ by any complex manifold transversal to $D$.
\end{remark}

We need the next proposition in our dimension counting argument.

\begin{prop}\label{decom}
Set $D_i = \{ p \in D \ | \ \textup{Rank}_{\mathbb{C}} \textup{Der}_X(-\log D)(p) = i \}$ for $i = 0, \ldots n-1$. Then for a locally quasi-homogeneous divisor $D$, $D_i$ is an analytic set of dimension at most $i$. 
\end{prop}

\begin{proof}
For locally quasi-homogeneous free divisors, the result is immediate. It is known that locally quasi-homogeneous free divisors are Koszul free and Koszul free divisors are holonomic (the logarithmic strata are locally finite) in the sense of Saito \cite{MR1931962}. According to \cite{MR586450} $dim_\mathbb{C}D_i \leq i$ for holonomic divisors.

For locally quasi-homogeneous divisors without the freeness assumption, the author is not able to find a direct reference. To check the conclusion of the proposition directly we can choose a local isomorphism $(D,p) \cong (D' \times \mathbb{C}^{i}, (0,0))$ at $p \in D_i$ with $D'$ a quasi-homogeneous divisor in an open subset of $\mathbb{C}^{n-i}$. We see that this step reduces the question to the case $i=0$, which is immediate because of the existence of the local Euler vector field.
 
\end{proof}

So far we have shown if the hyperplane $H \subset \Pbb^N$ cuts both $X$ and the locally quasi-homogenous free divisor $D$ transversally, then we can produce a nonsingular $X \cap H$ with locally quasi-homogeneous free diviosr $D \cap H$. To obtain the Bertini type result, we only need to count the dimension of the hyperplanes which fail to intersect $D$ or $X$ transversally. The set of  ``bad" hyperplanes which fail to cut $X$ transversally has dimension $N-1$: the result of the classical Bertini theorem. We now show:

\begin{prop}
Let $D$ be a locally quasi-homogeneous divisor (or more generally a holonomic divisor). For any $i$, $0 \leq i \leq n-1$, the set of hyperplanes that fail to intersect $D$ transversally at points from $D_i$ has dimension at most $N-1$.
\end{prop}

\begin{proof}
Denote by $P^N$ the projective space of hyperplanes in $\Pbb^N$. Consider the subspace $W_i$ of $D_i \times P^N$ consisting of pairs $(p, H)$ such that $H$ does not meet $D$ transversally at $p$. Also set $\pi_{ij}$ the projection from $W_i$ to the $j$th factor, $j = 1, 2 $. For any $p \in D_i$, the fiber $\pi_{i1}^{-1}(p)$ consists of hyperplanes of $\Pbb^N$ containing $TD_p$. Here we identify a hyperplane with its tangent space at any of its points. Because $dim(TD_p) = i$, we get $dim(\pi^{-1}(p)) = N - i -1$. Combining with proposition ~\ref{decom}, we get $dim(W_i) \leq (N-1-i) + i = N-1$. So $dim(\pi_{i2}(W)) \leq N-1$.
\end{proof}

\begin{corol}
The general hyperplanes in $\Pbb^N$ intersect $X$ and $D$ transversally at the same time. So the local quasi-homogeneity and the freeness of divisors is preserved by intersecting with general hyperplanes.
\end{corol}

\begin{proof}
$dim(\cup \pi_{i2}(W_i)) \leq N-1$
\end{proof}

As an application of the ideas discussed in this section, we prove a stronger version of Teissier's idealistic Bertini's theorem in the context of locally quasi-homogeneous divisors. Recall the following version of idealistic Bertini theorem stated in \cite{MR2904577} Lemma 30. Let $X=\Pbb^n$ and $D$ is a hypersurface in $X$. For a sufficient general hyperplane $H$ of $X$, the ideal of $D^{s} \cap H$ is integral over the ideal of $(D \cap H)^{s}$. Here $(\cdot)^{s}$ denotes the singular analytic subspace (or subscheme) of the given analytic space (or scheme). For a hypersurface with a local equation $f$, the ideal of the singular subspace is locally generated by all partial derivatives of $f$ as well as $f$ itself. In Teissier's original approach, however, the equation of the hypersurface $f$ did not appear in the definition of the singular analytic subspace. The seemingly difference occurs because Teissier treated the singular subspace as a subspace of $D$ whereas we treat the singular subspace as a subspace of $X$. For a general analytic space (or scheme), the ideal of the singular subspace can be described in terms of the fitting ideals of the sheaf of K\"{a}hler differentials. Note in particular the idealistic Bertini theorem implies that $D^{s} \cap H$ and $(D \cap H)^{s}$ have identical underlying topological spaces, although they might not be identical as ringed spaces.

\begin{corol}\label{bertini}
Let $X \subset \Pbb^N$ be an $n$-dimensional nonsingular complex projective variety and $D$ a locally quasi-homogeneous divisor in $X$. Let $H$ be a sufficient general hyperplane of $\Pbb^N$. Then $D^{s} \cap H = (D \cap H)^{s}$ as analytic subspaces (or subschemes) of $D$. 
\end{corol}

\begin{proof}
By the previous corollary we know that a sufficient general hyperplane intersect $D$ transversally. Thus locally we reduce to the case $D= D' \times \Cbb$ and $H$ is ``perpendicular" to the trivial factor. It is clear then $D^{s} = (D')^{s} \times \Cbb$. So $D^{s} \cap H = (D')^{s} = (D \cap H)^s$. The statement about subschemes is obtained from GAGA principle.
\end{proof}

\section{ The third step of the proof}
In this section, if not otherwise specified, $X \subset \Pbb^N$ is a smooth complex projective variety, $D$ is a free divisor in $X$, $U$ is the complement of $D$ in $X$, and $H$ is a hyperplane in $\Pbb^N$ intersecting $X$ and $D$ transversally. Also denote by $X'$, $D'$ and $U'$ the intersection of $H$ with $X$, $D$ and $U$ respectively. We derive a formula relating $c(\textup{Der}_X(-\log D)) \cap [X]$ and $c(\textup{Der}_{X'}(-\log D')) \cap [X']$. Replacing $c(\textup{Der}_X(-\log D)) \cap [X]$ by $c_{SM}(1_U)$, $c(\textup{Der}_{X'}(-\log D'))\cap [X']$ by $c_{SM}(1_{U'})$, the same formula is true.

For closed embeddings of nonsingular varieties, we know the normal bundle on the subvariety is the quotient of the tangent bundles. It turns out in our current setting, the quotient of the sheaves of logarithmic vector fields also defines the normal bundle.

\begin{prop}
Denote $N$ the normal bundle on $X'$ to $X$, and $i: X' \to X'$ the closed embedding. We have the exact sequence of vector bundles
\begin{equation*}
0 \to \textup{Der}_{X'}(-\log D') \to i^{*}(\textup{Der}_X(-\log D)) \to N \to 0
\end{equation*}
\end{prop} 

\begin{proof}
First note all the arrows in the sequence are naturally defined. The map $\textup{Der}_{X'}(-\log D') \to i^{*}(\textup{Der}_X(-\log D))$ is induced by $TX' \to i^{*}TX$, and $i^{*}(\textup{Der}_X(-\log D)) \to N$ is the restriction of $i^{*}TX \to N$ to $i^{*}(\textup{Der}_X(-\log D))$. This sequence of analytic coherent sheaves is exact because locally analytically the sheaf in the middle is the direct sum of the sheaves on the sides. The same sequence is also exact in the category of algebraic coherent sheaves as the GAGA principle applies.
\end{proof}

Taking Chern classes of this exact sequence and then pushing forward to $X$, we get:

\begin{equation*}
\begin{split}
i_*\left(c(\textup{Der}_{X'}(-\log D')) \cdot c(N) \cap[X']\right) &= c(\textup{Der}_X(-\log D)) \cap i_*[X'] \\
                                                                                                          &= c_1(\mathscr{O}_{X}(1)) \cdot c(\textup{Der}_X(-\log D)) \cap [X] \\
\end{split}
\end{equation*} 

As we have advertised before, there is a similar formula for CSM classes. For its proof, see proposition 2.6 in \cite{aluffi}.

\begin{prop}\label{csm}
\begin{equation*}
i_*\left(c(N) \cap c_{SM}(1_{U'})\right) = c_1(\mathscr{O}_X(1)) \cap (c_{SM}(1_U))
\end{equation*}
\end{prop}

\begin{remark}
Under the basic setting of this section, $X'$ is a hyperplane section of $X$, so the normal bundle $N \cong \mathscr{O}_{X'}(1)$.
\end{remark}

\section{Coda: The fourth step of the proof}
With all the preparations in previous sections, we can now prove theorem ~\ref{main} by an induction of the dimension of $X$.

\begin{proof}[Proof of theorem ~\ref{main}]

Let $X$ be a nonsingular complex projective curve, $D$ be a reduced effective divisor in $X$. Since in this case $D$ itself is nonsingular, therefore normal crossing, thus the classes in equation~\eqref{formula} are rationally equivalent\cite{MR1697199}. Alternatively, we can get the same conclusion by a direct computation. If $D= \sum P_i$, then $c_{SM}(U) = c(TX) \cap [X] - \sum [P_i]$. To calculate the Chern class of the logarithmic vector fields, we use the following exact sequences:
\begin{equation*}
0 \to \textup{Der}_X(-\log D) \to TX \to \mathscr{O}_D(D) \to 0
\end{equation*}
\begin{equation*}
0 \to \mathscr{O}_X \to \mathscr{O}_X(D) \to \mathscr{O}_D(D) \to 0
\end{equation*}
Taking Chern classes of these exact sequences we get $c(\textup{Der}_X(-\log D)) \cap [X] = (c(TX) \cdot c(\mathscr{O}_X(D)^{-1}) \cap [X] = c(TX) \cap [X] - \sum [P_i]$.

Assume the theorem is proved for all $k-1$ dimensional nonsingular complex projective varieties, we now finish the inductive step. Let $X$ be a $k$ dimensional nonsingular complex projective variety, $D$ be a locally quasi-homogeneous free divisor in $X$, $U$ be the complement of $D$ in $X$, and $H$ is a hyperplane of the ambient projective space intersecting both $X$ and $D$ transversally. Denote by $X'$, $D'$ and $U'$  respectively the intersections of the corresponding spaces with $H$.  

According to proposition ~\ref{degree}, the degrees of the CSM class and the Chern class of the logarithmic vector fields are always the same. To establish other numerical relations, we use the following equalities:
\begin{equation*}
\begin{split}
   & \int \! c_1(\mathscr{O}_X(1))^i \cap c_{SM}(1_U) \\
 =& \int \! c_1(\mathscr{O}_X(1))^{i-1} \cap \Big(c_1(\mathscr{O}_X(1)) \cap c_{SM}(1_U)\Big) \quad (i \geq 1) \\
                                                                                       =& \int \! c_1(\mathscr{O}_X(1))^{i-1} \cap i_*\Big(c(\mathscr{O}_{X'}(1)) \cap c_{SM}(1_{U'})\Big) \quad \text{(proposition ~\ref{csm})}\\
                                                                                       =& \int \! c_1(\mathscr{O}_X(1))^{i-1} \cap \Big(c(\mathscr{O}_X(1) \cap i_*c_{SM}(1_U')\Big) \quad \text{(projection formula)} \\
                                                                                       =& \int \! c_1(\mathscr{O}_X(1))^{i-1} \cap i_*c_{SM}(1_U') + \int \! \! c_1(\mathscr{O}_X(1))^i \cap i_*c_{SM}(1_U') \\
                                                                                       =& \int \! c_1(\mathscr{O}_{X'}(1))^{i-1} \cap c_{SM}(1_U') + \int \! \! c_1(\mathscr{O}_{X'}(1))^i \cap c_{SM}(1_U') \quad \text{(projection formula)}\\                                                                                   
\end{split}
\end{equation*}

Applying the analogous formula for logarithmic vector fields instead of proposition ~\ref{csm} we have:
\begin{equation*}
\begin{split}
   & \int \! c_1(\mathscr{O}_X(1))^i \cap \Big(c(\textup{Der}_X(-\log D) \cap [X] \Big) \\
 =& \int \! c_1(\mathscr{O}_{X'}(1))^{i-1} \cap \Big(c(\textup{Der}_{X'}(-\log D') \cap [X'] \Big) +\\
                                                                                                                                              & \int \! \! c_1(\mathscr{O}_{X'}(1))^i \cap \Big(c(\textup{Der}_{X'}(-\log D') \cap [X'] \Big) \quad (i \geq 1)
\end{split}
\end{equation*}

Now the theorem follows from inductive hypothesis.

\end{proof}

\begin{proof}[Proof of corollary ~\ref{prspace}]
For $X=\Pbb^n$, the morphism from $A_i(X)$ to $\Zbb$
\begin{equation*}
\alpha \mapsto \int \! \alpha \cap c_1(\mathscr{O}(1))^i
\end{equation*}
is an isomorphism.
\end{proof}

\bibliographystyle{alpha}
\bibliography{liaobib}

\begin{thebibliography}{CJNMM96}

\bibitem[Alu]{aluffi}
Paolo Aluffi.
\newblock Euler characteristics of general linear sections and polynomial chern
  classes.
\newblock arXiv:1207.6638. To appear in Rendiconti del Circolo Matematico di
  Palermo.

\bibitem[Alu99]{MR1697199}
Paolo Aluffi.
\newblock Chern classes for singular hypersurfaces.
\newblock {\em Trans. Amer. Math. Soc.}, 351(10):3989--4026, 1999.

\bibitem[Alu12a]{MR2928931}
Paolo Aluffi.
\newblock Chern classes of free hypersurface arrangements.
\newblock {\em J. Singul.}, 5:19--32, 2012.

\bibitem[Alu12b]{hyparr}
Paolo Aluffi.
\newblock Grothendieck classes and {C}hern classes of hyperplane arrangements.
\newblock {\em Int. Math. Res. Not.}, 2012.

\bibitem[CJNMM96]{MR1363009}
Francisco~J. Castro-Jim{\'e}nez, Luis Narv{\'a}ez-Macarro, and David Mond.
\newblock Cohomology of the complement of a free divisor.
\newblock {\em Trans. Amer. Math. Soc.}, 348(8):3037--3049, 1996.

\bibitem[CMNM02]{MR1931962}
Francisco Calder{\'o}n-Moreno and Luis Narv{\'a}ez-Macarro.
\newblock The module {$\mathscr{D}f^s$} for locally quasi-homogeneous free
  divisors.
\newblock {\em Compositio Math.}, 134(1):59--74, 2002.

\bibitem[CMNM09]{MR2500865}
F.~J. Calder{\'o}n~Moreno and L.~Narv{\'a}ez~Macarro.
\newblock On the logarithmic comparison theorem for integrable logarithmic
  connections.
\newblock {\em Proc. Lond. Math. Soc. (3)}, 98(3):585--606, 2009.

\bibitem[Fab]{faber}
Eleonore Faber.
\newblock Towards transversality of singular varieties: splayed divisors.
\newblock arXiv:1201.2186.

\bibitem[Ful84]{MR732620}
William Fulton.
\newblock {\em Intersection theory}, volume~2 of {\em Ergebnisse der Mathematik
  und ihrer Grenzgebiete (3) [Results in Mathematics and Related Areas (3)]}.
\newblock Springer-Verlag, Berlin, 1984.

\bibitem[Huh12]{MR2904577}
June Huh.
\newblock Milnor numbers of projective hypersurfaces and the chromatic
  polynomial of graphs.
\newblock {\em J. Amer. Math. Soc.}, 25(3):907--927, 2012.

\bibitem[Lan03]{MR1971155}
Adrian Langer.
\newblock Logarithmic orbifold {E}uler numbers of surfaces with applications.
\newblock {\em Proc. London Math. Soc. (3)}, 86(2):358--396, 2003.

\bibitem[Lia]{liao}
Xia Liao.
\newblock Chern classes for logarithmic vector fields.
\newblock Proceedings of the Hefei Conference on Singularity, arXiv:1201.6110.

\bibitem[Mac74]{MR0361141}
R.~D. MacPherson.
\newblock Chern classes for singular algebraic varieties.
\newblock {\em Ann. of Math. (2)}, 100:423--432, 1974.

\bibitem[MS01]{MR1843320}
Mircea Musta{\c{t}}{\v{a}} and Henry~K. Schenck.
\newblock The module of logarithmic {$p$}-forms of a locally free arrangement.
\newblock {\em J. Algebra}, 241(2):699--719, 2001.

\bibitem[OT92]{MR1217488}
Peter Orlik and Hiroaki Terao.
\newblock {\em Arrangements of hyperplanes}, volume 300 of {\em Grundlehren der
  Mathematischen Wissenschaften [Fundamental Principles of Mathematical
  Sciences]}.
\newblock Springer-Verlag, Berlin, 1992.

\bibitem[Sai71]{MR0294699}
Kyoji Saito.
\newblock Quasihomogene isolierte {S}ingularit\"aten von {H}yperfl\"achen.
\newblock {\em Invent. Math.}, 14:123--142, 1971.

\bibitem[Sai80]{MR586450}
Kyoji Saito.
\newblock Theory of logarithmic differential forms and logarithmic vector
  fields.
\newblock {\em J. Fac. Sci. Univ. Tokyo Sect. IA Math.}, 27(2):265--291, 1980.

\bibitem[Sch03]{MR2031639}
J{\"o}rg Sch{\"u}rmann.
\newblock {\em Topology of singular spaces and constructible sheaves},
  volume~63 of {\em Instytut Matematyczny Polskiej Akademii Nauk. Monografie
  Matematyczne (New Series) [Mathematics Institute of the Polish Academy of
  Sciences. Mathematical Monographs (New Series)]}.
\newblock Birkh\"auser Verlag, Basel, 2003.

\bibitem[Tei77]{MR0568901}
Bernard Teissier.
\newblock The hunting of invariants in the geometry of discriminants.
\newblock In {\em Real and complex singularities ({P}roc. {N}inth {N}ordic
  {S}ummer {S}chool/{NAVF} {S}ympos. {M}ath., {O}slo, 1976)}, pages 565--678.
  Sijthoff and Noordhoff, Alphen aan den Rijn, 1977.

\bibitem[XY96]{MR1385285}
Yi-Jing Xu and Stephen S.-T. Yau.
\newblock Micro-local characterization of quasi-homogeneous singularities.
\newblock {\em Amer. J. Math.}, 118(2):389--399, 1996.

\end{thebibliography}
  
\end{document}